\documentclass[12pt,twoside,reqno]{amsart}

\allowdisplaybreaks
\newtheorem{thm}{Theorem}[section]
\newtheorem{lemma}{Lemma}[section]
\newtheorem{rem}{Remark}[section]

\theoremstyle{definition}

\theoremstyle{remark}

\newcommand{\R}{{\mathbb R}}
\newcommand{\N}{{\mathbb N}}

\newcommand{\noi}{\noindent}

\numberwithin{equation}{section}

\begin{document}
\title[K-S on grooves]	
{	Global  Solutions for the   Kuramoto-Sivashinsky equation posed on unbouded 3D grooves
}
\author{N. A. Larkin}

\address
{
	Departamento de Matem\'atica, Universidade Estadual
	de Maring\'a, Av. Colombo 5790: Ag\^encia UEM, 87020-900, Maring\'a, PR, Brazil
}

\thanks
{\tiny{Departamento de Matemática, Universidade Estadual de Maringá, Av. Colombo 5790, 87020-900,
		Maringá, Parana, Brazil; email: nlarkine@uem.br}
}
\bigskip

\email{ nlarkine@uem.br;nlarkine@yahoo.com.br }
\date{}

\begin{abstract}Initial boundary value problems for the three dimensional Kuramoto-Sivashinsky equation posed on unbounded 3D grooves were considered. The existence and uniqueness of  global  strong solutions  as well as their exponential decay  have been established.
\end{abstract}

\maketitle

\section{Introduction}\label{introduction}
This work concerns the existence and uniqueness of global strong  solutions as well  as  exponential decay rates of solutions to initial-boundary value problems for the three dimensional Kuramoto-Sivashinsky  equation (K-S):
\begin{align}
	& \phi_t+\Delta^2 \phi+ \Delta \phi +\frac{1}{2}|\nabla \phi|^2=0.
\end{align}
Here $\Delta$ and $\nabla$ are the Laplacian and the gradient  in  $\R^3.$
In \cite{kuramoto}, Kuramoto studied the turbulent phase waves  and Sivashinsky in \cite{sivash} obtained an asymptotic equation which modeled the evolution of a disturbed plane flame front. See also \cite{Cross,cuerno}. In \cite{Iorio, Guo,cousin,feng,Larkin2,temam1,temam2,zhang},  mathematical  results on  initial and initial boundary value problems for one dimensional (1.1)  are presented, see references  there for more information. Multidimesional problems for various  types of (1.1) can be found in \cite{kukavica,gramchev,Guo,molinet,temam1,sell,temam2} with some results on existence, regularity and nonlinear stability of solutions.\\
For three dimensions, (1.1) can be rewritten in the form of the following system:

 \begin{align}
	&(u_j)_t+\Delta^2 u_j+\Delta u_j +\frac{1}{2}\sum_{i=1}^3(u^2_i)_{x_j}=0,   \\
	&(u_j)_{x_i}=(u_i)_{x_j},\;\;i\ne j;\;\;i,j,=1,2,3,
	\end{align}
where $u_j=\phi_{x_j}.$ 
First essential problem that arises while one studies either (1.1) or (1.2)-(1.3) is a destabilizing effect of 
$\Delta u_j $ that may be damped by a dissipative term $\Delta^2 u_j$ provided a domain has some specific properties. Naturally, so called "thin domains"  here appear where some dimensions are small while the others may be arbitrarily large.
Second essential problem is presence of semilinear interconnected terms in (1.2). This does not allow to obtain  first estimate independent of solutions and leads to a connection between geometric properties of a domain and initial data.\\
Our work has the following structure: Chapter I is Introduction. Chapter 2 contains notations and auxiliary facts. In Chapter 3, formulation of an initial boundary value problem for (1.2)-(1.3) posed on a unbounded groove  is given. The existence of a global strong  solution, exponential decay of the $L^2$-norm  have been established.   Chapter 4  contains conclusions.

\section{Notations and Auxiliary Facts}

Let $\Omega$ be a domain in $\R^3$ and $x = (x_1,x_2,x_3) \in \Omega$. We use the standard notations of Sobolev spaces $W^{k,p}$, $L^p$ and $H^k$ for functions and the following notations for the norms \cite{Adams, Brezis}:
for scalar functions $f(x,t)$\\
$$\R^+=\{t\in\R^1;\;t\geq 0\},\;\;\| f \|^2 = \int_{\Omega} | f |^2d\Omega, \;\; \| f \|_{L^p(\Omega)}^p = \int_{\Omega} | f  |^p\, d\Omega,$$
$$\| f \|_{W^{k,p}(\Omega)}^p = \sum_{0 \leq \alpha \leq k} \|D^\alpha f \|_{L^p(\Omega)}^p, \hspace{1cm} \| f \|_{H^k(\Omega)} = \| f \|_{W^{k,2}(\Omega)}.$$

When $p = 2$, $W^{k,p}(\Omega) = H^k(\Omega)$ is a Hilbert space with the scalar product 
$$((u,v))_{H^k(\Omega)}=\sum_{|j|\leq k}(D^ju,D^jv),\;
\|u\|_{L^{\infty}(\Omega)}=ess\; sup_{\Omega}|u(x)|.$$
We use a notation $H_0^k(\Omega)$ to represent the closure of $C_0^\infty(\Omega)$, the set of all $C^\infty$ functions with compact support in $\Omega$, with respect to the norm of $H^k(\Omega)$.

\begin{lemma}[Steklov's Inequality \cite{steklov}] Let $v \in H^1_0(0,L).$ Then
	\begin{equation}\label{Estek}
	\frac {\pi^2}{L^2}\|v\|^2 \leq \|v_x\|^2.
	\end{equation}
\end{lemma}

\begin{lemma}
	[Differential form of the Gronwwall Inequality]\label{gronwall} Let $I = [t_0,t_1]$. Suppose that functions $a,b:I\to \R$ are integrable and a function $a(t)$ may be of any sign. Let $u:I\to \R$ be a differentiable function satisfying
	\begin{equation}
		u_t (t) \leq a(t) u(t) + b(t),\text{ for }t \in I\text{ and } \,\, u(t_0) = u_0,
	\end{equation}
	then
	$$u(t) \leq u_0 e^{ \int_{t_0}^t a(\tau)\, d\tau } + \int^t_{t_0} e^{\int_{t_0}^s a(r) \, dr} b(s) ds$$\end{lemma}
The next Lemmas will be used in  estimates:

\begin{lemma}[See \cite{Lady}, Theorem 7.1; \cite{Temam}, Lemma 3.5.] 	 Let $v \in H_0^1(\Omega)$ and $n = 3$, then	
	\begin{equation}
		\label{3.3}\| v \|_{L^4(\Omega)} \leq 2^{1/2} \| v \|^{1/4} \| \nabla v \|^{3/4}.
	\end{equation}
\end{lemma}

\section{ K-S system posed on a groove}

Define a groove  $$D=\{x\in\R^3;\;x_1\in \R^1,\; x_2\in  (0,B>0),\; x_3 \in \R^+\},\;Q_t=(0,t)\times D.$$
 \begin{lemma}
	Let $f\in  H^2_0(D).$ Then
	\begin{align}
		&a\|f\|^2\leq\|\nabla f\|^2,\;\;a^2\|f\|^2\leq \|\Delta f\|^2,\;\;a\|\nabla f\|^2\leq \|\Delta f\|^2,\\
		&\text{where} \; a=\frac{\pi^2}{B^2}.
	\end{align}	
\end{lemma}

\begin{proof} By definition,
	$$\|\nabla f\|^2=\sum_{i=1}^3\|f_{x_i}\|^2.$$
	Making use of Steklov`s inequalities, we get
	$$\|\nabla f\|^2\geq \|f_{x_2}\|^2\geq\frac{\pi^2}{B^2}\|f\|^2=a\|f\|^2.$$
	On the other hand,
	$$a\|f\|^2\leq \|\nabla f\|^2 =-\int_{D_h}f\Delta fdx\leq \|\Delta f\|\|f\|.$$
	This implies
	$$a\|f\|\leq \|\Delta f\|\;\;\text{and} \;\;a^2\|f\|^2\leq \|\Delta f\|^2.$$
	Consequently, \;$a\|\nabla f\|^2\leq \|\Delta f\|^2.$ 
	Proof of Lemma 3.1 is complete.
\end{proof}
In $Q_{t}$ consider the following initial boundary value problem:

\begin{align}
	(u_j)_t+\Delta^2 u_j+\Delta u_j +\frac{1}{2}\sum _{i=1}^3(u^2_i)_{x_j}=0,&\\
	(u_i)_{x_j}=(u_j)_{x_i},\; j\ne i,\;\; i,j=1,...,3;&\\
	u_j|_{\partial D_j}=\frac{\partial}{\partial N} u_j|_{\partial D_z}=0,\; t>0,&\\
	u_j(x,0)=u_{j0}(x),\;j=1,...,3,\;\;x \in D.
\end{align}
where $\frac{\partial}{\partial N}$ is an exterior normal derivative on $\partial D$.

\begin{thm} Let 
	\begin{align} B<\pi,\; a=\frac{\pi^2}{B^2},\;\;\theta=1-\frac{1}{a}>0.
	\end{align}
	Given $u_{j0}\;\in H^4(D)\cap H^2_0(D),\;j=1,2,3$ such that
	\begin{align}
		\theta-\frac{48}{\theta a^{3/2}}\sum_{i=1}^3\|u_{i0}\|^2>0.
	\end{align}
	Then the problem (3.3)-(3.6) has a unique strong solution
	$$u_j\;\in L^{\infty}(\R^+;H^2_0(D)),\; \;\Delta^2 u_j\; \in L^{\infty}(\R^+;L^2(D));$$
	$$ u_{jt}\; \in  L^{\infty}(\R^+;L^2(D))\cap L^2(\R^+;H^2_0(D),\;j=1,2,3.$$
	Moreover, $u_j$ satisfy the following inequality
	\begin{align}
		\sum_{j=1}^3\|u_j\|^2(t)+ \frac{\theta}{2}\int_0^t\Big(\sum_{j=1}^3\|\Delta u_j\|^2(\tau)\Big)d\tau\leq \sum_{j=1}^3 \|u_{j0}\|^2,&\\
		\sum_{j=1}^3\|u_j\|^2(t)\leq \Big[\sum_{j=1}^3 \|u_{j0}\|^2\Big]\exp\{-\frac{a^2\theta t}{2}\}.
	\end{align}
	\begin{align}\sum_{j=1}^3\|u_{jt}\|^2(t)\leq
		\Big(\sum_{j=1}^3\|u_{jt}\|^2(0)\Big)\exp\{-\frac{a^2\theta t}{2}\},
	\end{align}
	\begin{align}\sum_{j=1}^3\|u_{jt}\|^2(t)+
		\frac{\theta}{2}\int_0^t\Big(\sum_{j=1}^3\|\Delta u_{j\tau}\|^2(\tau)\Big)d\tau&\notag\\\leq\sum_{j=1}^3\|u_{jt}\|^2(0),
	\end{align}
	where, $$ \|u_{jt}\|^2(0)\|\leq C(\|u_{0j}\|_W),\;j=1,2,3.$$
\end{thm}
\begin{proof}
	
	Define the space $W=H^4(D)\cap H^2_0(D)$ and let $\{w_i(x), \; i\in \N\}$ be a countable dense set in $W$.
	We can construct approximate solutions to (3.3)-(3.6) in the form
	$$u^N_j(x,t)=\sum_{i=1}^N g_i^j(t)w_i(x);\;j=1,2,3.$$
	Unknown functions $g_i^j(t)$\;\;satisfy the following initial problems:
	\begin{align}\frac{d}{dt}(u^N_j,w_j)(t)+(\Delta u^N_j,\Delta w_j)(t)-(\nabla u^N_j, \nabla w_j)(t)&\notag\\
		-\frac{1}{2}\sum_{i=1}^3((u^N_i)^2,(w_j)_{x_j})(t)=0,&\\
		g_i^j(0)=g_{i0}^j, \;\j=1,2,3;\;\;i=1,2,... .
	\end{align}
By Caratheodory`s existence theorem, there exist solutions of (3.13)-(3.14) at least locally in $t$.
All the estimates we will prove will be done on smooth solutions of (3.3)-(3.6). Naturally, the same estimates are true also for approximate solutions $u^N_j.$\\ 
\noi {\bf Estimate I }Mltiply (3.3) by $2u_j$  to obtain 

\begin{align}\frac {d}{dt}\|u_j\|^2(t)+2\|\Delta u_j\|^2(t)-2\|\nabla u_j\|^2(t)&\notag\\
	-\sum_{i=1}^3(u_i^2,(u_j)_{x_j})(t)=0,\;j=1,2,3.
\end{align}

Since $a\|\nabla u_j\|^2(t)\leq \|\Delta u_j\|^2(t)$ , taking into account  (3.7), we get

\begin{align}\frac {d}{dt}\|u_j\|^2(t)+\theta\|\Delta u_j\|^2(t)+\theta\|\Delta u_j\|^2(t)&\notag\\
	-\sum_{i=1}^3(u_i^2,(u_j)_{x_j})(t)\leq 0.
\end{align}

Making use of Lemmas 2.3, 3.1, we estimate

$$I=\sum_{i=1}^3(u_i^2,(u_j)_{x_j})\leq \sum_{i=1}^3\|u_i\|\|u_i\|_{L^4(D)}\|(u_j)_{`x_j}\|_{L^4(D)}$$
$$\leq \sum_{i=1}^32\|u_i\|\|u_i\|^{1/4}\|\nabla u_i\|^{3/4}\|\nabla u_j\|^{1/4}\|\Delta u_j\|^{3/4}$$
$$\leq 2\sum_{i=1}^3\frac{1}{a^{1/4}}\|u_i\|\|\nabla u_i\|\|\Delta u_j\| \leq \epsilon \|\Delta u_j\|^2$$
$$+ \frac{1}{\epsilon}\Big(\sum_{i=1}^3\frac{1}{a^{1/4}}\|u_i\|\|\nabla u_i\|\Big)^2\leq\epsilon \|\Delta u_j\|^2$$
$$+ \frac{4}{a^{1/2}\epsilon}\sum_{i=1}^3\\|u_i\|^2\|\nabla u_i\|^2\leq \epsilon \|\Delta u_j\|^2+ \frac{4}{a^{3/2}\epsilon}\sum_{i=1}^3\\|u_i\|^2\|\Delta u_i\|^2,$$
where $\epsilon$ is an arbitarary positive number.
Taking $2\epsilon=\theta,$ substituting $I$ into (3.16) and summing up over $j=1,2,3$, we get

\begin{align}\frac {d}{dt}\sum_{j=1}^3\|u_j\|^2(t)+\frac{\theta}{2}\sum_{j=1}^3\|\Delta u_j\|^2(t)&\notag\\
	+\Big[\theta-\frac{24}{\theta a^{3/2}}\sum_{i=1}^3\|u_i\|^2(t)\Big]\sum_{j=1}^3\|\Delta u_j\|^2(t)	\leq 0.
\end{align}

Condition (3.8), positivity of the second term in (3.17) and standard arguments guarantee that
\begin{equation}\Big[\theta-\frac{24}{\theta a^{3/2}_z}\sum_{i=1}^3\|u_i\|^2(t)\Big]>0 , \;\;t>0.
\end{equation}
This transforms (3.17) into the following inequality:

\begin{align}\frac {d}{dt}\sum_{j=1}^3\|u_j\|^2(t)+\frac{\theta}{2}\sum_{j=1}^3\|\Delta u_j\|^2(t)	\leq 0.
\end{align}

Integrating (3.19), we obtain

\begin{align}
	\sum_{j=1}^3\|u_j\|^2(t+\int_0^t\frac{\theta}{2}\sum_{j=1}^3\|\Delta u_j\|^2(\tau)\,d\tau \leq \sum_{j=1}^3\|u_{j0}\|^2.
\end{align}
On the other hand, Lemma 3.1 allows us to rewrite (3.19) in the form
\begin{align}\frac {d}{dt}\sum_{j=1}^3\|u_j\|^2(t)+\frac{a^2\theta}{2}\sum_{j=1}^3\| u_j\|^2(t)	\leq 0.
\end{align}
From this, (3.10) follows.

\noi {\bf Estimate II }

Differentiate (3.3) with respect to $t$, then multiply the results respectively  by $2(u_j)_t$  to get
\begin{align}
	\frac{d}{dt}\|u_{jt}\|^2(t)+2\|\Delta u_{jt}\|^2(t)-2\|\nabla u_{jt}\|^2(t)&\notag\\
	=2\sum_{i=1}^3(u_iu_{it},(u_j)_{x_{jt}})(t),\;j=1,2,3.
\end{align}
Making use of Lemmas 2.4 and 3.1, we estimate
\begin{align*}I=2\sum_{i=1}^3(u_iu_{it},(u_j)_{x_{jt}})(t)
	\leq 2\sum_{i=1}^3\|u_i\|(t)\|u_{it}\|_{L^4(D)}(t)\|\nabla u_{jt}\|_{L^4(D)}(t)
\end{align*}

$$
\leq 4\sum_{i=1}^3\|u_i\|\|u_{it}\|^{1/4}\|\nabla u_{it}\|^{3/4}\|\nabla u_{it}\|^{1/4}\|\Delta u_{it}\|^{3/4}$$
$$\leq 4\sum_{i=1}^3\frac{1}{a^{1/4}}\|u_i\|\|\nabla u_{it}\|\|\Delta u_{it}\|\leq \epsilon\|\Delta u_{it}\|^2$$
$$+ \frac{4}{\epsilon}\Big(\sum_{i=1}^3\frac{1}{a^{1/4}}\|u_i\|\|\nabla u_{it}\|\Big)^2\leq\epsilon\|\Delta u_{it}\|^2$$$$+\frac{16}
{a^{1/2}\epsilon}\sum_{i=1}^3\|u_i\|^2\|\nabla u_{it}\|^2\leq\epsilon\|\Delta u_{it}\|^2+\frac{16}{a^{3/2}\epsilon} \sum_{i=1}^3\|u_i\|^2\|\Delta u_{it}\|^2.$$
Taking $2\epsilon=\theta,$ substituting $I$ into (3.22) and summing up over $j=1,2,3$, we get

\begin{align}\frac {d}{dt}\sum_{j=1}^3\|u_{jt}\|^2(t)+\frac{\theta}{2}\sum_{j=1}^3\|\Delta u_{jt}\|^2(t)&\notag\\
	+\Big[\theta-\frac{48}{\theta a^{3/2}}\sum_{i=1}^3\|u_i\|^2(t)\Big]\sum_{j=1}^3\|\Delta u_{jt}\|^2(t)\leq 0.
\end{align}

Taking into account (3.7), rewrite (3.23) in the form

\begin{align}\frac {d}{dt}\sum_{j=1}^3\|u_{jt}\|^2(t)+\frac{\theta}{2}\sum_{j=1}^3\|\Delta u_{jt}\|^2(t)\leq 0.
\end{align}

This implies

\begin{align}\sum_{j=1}^3\|u_{jt}\|^2(t)+\frac{\theta}{2}\int_0^t\Big(\sum_{j=1}^3\|\Delta u_{j\tau}\|^2(\tau)d\tau
	\leq \sum_{j=1}^3\|u_{jt}\|^2(0),
\end{align}
where $\|u_{jt}\|^2(0)\leq C(\|u_{0j}\|_W).$ Making use of Lemma 3.1, rewrite (3.24) as

\begin{align}\frac {d}{dt}\sum_{j=1}^3\|u_{jt}\|^2(t)+\frac{a^2\theta}{2}\sum_{j=1}^3\| u_{jt}\|^2(t)\leq 0.
\end{align}

Integrating this, we find

\begin{align} \sum_{j=1}^3\|u_{jt}\|^2(t)\leq \Big(\sum_{j=1}^3\|u_{jt}\|^2(0)\Big)\exp\{-\frac{a^2\theta}{2}t\}.
	\end{align}

Estimates (3.20) and (3.25) imply that $u_j \in L^{\infty}(\R^+;H^2_0)(D)),$$$u
_{jt}\in L^{\infty}(\R^+;L^2(D))\cap L^2(\R^+;H^2_0(D)),\;j=1,2,3.$$

These inequalities guarantee the existence of  strong soluitons to (3.3)-(3.6) $\{u_j(x,t)\}$ satisfying  the following integral identities:
\begin{align} ((u_j)_t,\phi)(t)+(\Delta u_j,\Delta \phi)(t)+(\Delta u_j,\phi)(t)&\notag\\
	-\frac{1}{2}\sum_{i=1}^3(u_i)^2,\phi_{x_j})(t)=0,\;t>0,
\end{align}
where $\phi(x,y)$ is an arbitrary function from $H^2_0(D).$\\
We can rewrite (3.28) in the form

\begin{align*}(\Delta u_j,\Delta \phi)(t)=-([(u_j)_t+\Delta u_j-\sum_{i=1}^3u_i(u_i)_{x_j}] ,\phi)(t).
\end{align*}

It follows from here and estimates above that $$\Delta^2 u_j \in L^{\infty}(\R^+;L^2(D)),\;j=1,2,3.$$

This proves the existence part of Theorem 3.1.
\begin{lemma} The strong solutions of (3.3)-(3.6) is unique.
	\end{lemma}
 \begin{proof} Let $u_j$ and $v_j, \;j=1,2,3,$ be two distinct solutions to (3.3)-(3.6). Denoting $w=u_j-v_j$, we come to the following system:
 	
 	\begin{align} \frac{d}{dt}\|w_j\|^2(t)+2\|\Delta w_j\|^2(t)-2\|\nabla w_j\|^2(t)&\notag\\
 		=\sum_{i=1}^4(\{u_i+v_i\}w_i,(w_j)_{x_j})(t),&\\
 		(w_i)_{x_j}=(w_i)_{x_j}, \;i\ne j,&\\
 		w_j|_{\partial D_1}=\frac{\partial}{\partial N} w_j|_{\partial D_1}=0,\; t>0,&\\
 		w_j(x,0)=0,\;\;j=1,2,3.	
 	\end{align}
 	
 	Making use of Lemmas 2.3, 3.1, we estimate
 	$$
 	I=(\{u_i+v_i\}w_i,(w_j)_{x_j}) \leq \|(w_j)_{x_j}\|_{L^4(D_z)}\|w_i\|\|u_i+v_i\|_{L^4(D_z)}$$
 $$\leq 2\|w_i\|\|(w_j)_{x_j}\|^{1/4}\|\nabla (w_j)_{x_j}\|^{3/4}\|u_i+v_i\|^{1/4}\|\nabla (u_i+v_i)\|^{3/4}$$
 $$\leq \frac{2}{a^{1/4}}\|w_i\|\|\Delta (w_j)\|\|\nabla (u_i+v_i)\|$$
 $$\leq \epsilon\|\Delta (w_j)\|^2+ \frac{1}{\epsilon a^{3/2}}\|w_i\|^2\|\Delta (u_i+v_i)\|^2.$$
 
Substituting $I$ into (3.29), taking \;$2\epsilon=\theta$ and summing over $j=1,2,3,$ we get
 
\begin{align*} 
 \frac{d}{dt}\sum_{j=1}^3\|w_j\|^2(t)+\theta\sum_{j=1}^3\|\Delta w_j\|^2(t)&\notag\\ 
  \leq C\Big(\sum_{i=1}^3\{\|\Delta u_i\|^2+\|\Delta v_i\|^2\}\sum_{j=1}^3\|w_j\|^2(t)\Big).
 \end{align*}
Due to (3.20), $\|\Delta u_i\|^2(t)+\|\Delta v_i\|^2(t)\in L^1(\R^+),\;i=1,2,3.$\\
Applying Lemma 2.2, we obtain that
$$ \|w_j\|(t)\equiv 0,\; \;t>0,\;\;j=1,2,3.$$
This proves Lemma 3.2 and consequently Theorem 3.1. 
\end{proof} 
\begin{rem} We can define other grooves:
$$ D_{x_3}=\{ x=(x_1,x_2,x_3);\; x_1\in \R^1,\;x_2 \in \R^+,\; x_3\in (0,L_3>0),\}$$	
$$ D_{x_1}=\{ x=(x_1,x_2,x_3);\; x_1\in(0,L_1>0),\; x_2\in \R^1,\;x_3\in \R^+\}$$
and  obtain results similar to ones of Theorem 3.1.
\end{rem}

\section{ Conclusions}

In this work,  we studied initial boundary value problems for the three dimensional Kuramoto-Sivashinsky system (1.1) posed on  unbounded grooves. We defined a set of admissible domains which eliminate  destabilizing effects of terms $\Delta u_j$ by dissipativity of $\Delta^2 u_j.$   Since these problems do not admit the first a priori estimate independent of $t$ and solutions, in order to prove the existence of global  solutions, we put conditions  connecting geometrical properties of domains with initial data. We proved  the existence and uniqueness of a strong solution as well as exponential decay of $L^2$-norms.

\end{proof}
\section*{Conflict of Interests}

The author declares that there is no conflict of interest regarding the publication of this paper.

\medskip

\bibliographystyle{torresmo}
\end{document}